\documentclass{amsart} 
\usepackage{amssymb,latexsym}
\usepackage[utf8]{inputenc}
\usepackage[OT1]{fontenc}
\usepackage{amsmath}
\usepackage{graphicx}
\usepackage{enumerate}
\usepackage{newlattice}

\newcommand{\nfive}{\SN 5}

\newcommand{\ab}{L ^ {a,b}}
\newcommand{\bgaab}{\bga ^ {a,b}}
\newcommand{\sdmeet}{\textup{SD}${}_\mm$}

\theoremstyle{plain}
\newtheorem{theorem}{Theorem}
\newtheorem{lemma}[theorem]{Lemma}
\newtheorem{claim}[theorem]{Claim}

\newtheorem{corollary}[theorem]{Corollary}
\theoremstyle{definition}
\newtheorem{definition}[theorem]{Definition}

\begin{document}
\title[Infinite semidistributive lattices]{Congruence lattices of infinite semidistributive lattices}

\author[G.\ Gr\"atzer]{George Gr\"atzer}
\email[G.\ Gr\"atzer]{gratzer@mac.com}
\urladdr{http://server.maths.umanitoba.ca/homepages/gratzer/}
\address{University of Manitoba}

\author{J.\,B. Nation}
\email[J.B. Nation]{jb@math.hawaii.edu}
\urladdr{ https://math.hawaii.edu/~jb/}
\address{University of Hawaii\\
   Honolulu, HI\\
   USA}

\date{\today}

\begin{abstract} 
An FN lattice $F$ is a simple, infinite, semidistributive lattice.
Its existence was recently proved by R. Freese and J.\,B. Nation.
Let $\mathsf{B}_n$ denote the Boolean lattice with $n$ atoms.
For a lattice $K$, let $K^+$ denote $K$ with a new unit adjoined.

We prove that the finite distributive lattices: 
$\mathsf{B}_0^+, \mathsf{B}_1^+,\mathsf{B}_2^+, \dots$
can be represented as congruence lattices of infinite semidistributive lattices. 
The case $n = 0$ is the Freese-Nation result, which is utilized in the proof.

We also prove some related representation theorems.
\end{abstract}

\maketitle

\section{Introduction}\label{S:Introduction}

Not every finite distributive lattice is isomorphic to the congruence lattice 
of~a~finite semidistributive lattice:  for example, the 3-element chain is not,
see K. Adaricheva, R. Freese and J.\,B. Nation \cite{AFN22}
and J.\,B. Nation \cite{JBN23}.
Also, the 2-element chain is the only simple finite semidistributive lattice.
It turns out that congruence lattices of infinite semidistributive lattices may have 
fewer restrictions.

An FN lattice $F$ is a simple, infinite, semidistributive lattice.
Its existence was proved in R. Freese and J.\,B. Nation \cite{FN21}.
We will use that such a lattice $F$ has no zero or unit,
because if it did, it would not be simple.
(If a nontrivial join semidistributive lattice $L$
has a greatest element $1$, then an easy argument using 
join semidistributivity shows that if $I$ is an ideal of $L$
maximal with respect to not containing $1$, then $I$ and its
complement are the blocks of a congruence of $L$.)

Let $\SB n$ denote the Boolean lattice with $n$ atoms.
For a lattice $K$, let $K^+$ denote~the lattice we obtain by adding a new unit 
to $K$ and let $K_+$ denote $K$ with a~new zero adjoined.

\begin{theorem}\label{T:main}
There is an  infinite semidistributive lattice $L_i$ such that 
the congruence lattice of\/ $\SL i$ is isomorphic to  $\SB i^+$
for $n = 0,1, 2,\dots$.  See \emph{Figure~\ref{F:sequence1}.}
\end{theorem}

For $n = 0$, this gives the two-element chain.
Note that for $n = 1$, Theorem~\ref{T:main} claims
that there is an infinite semidistributive lattice
with the three-element as the congruence lattice.

\begin{corollary}\label{C:main2}
There is an  infinite semidistributive lattice $L_i$ such that 
the congruence lattice of\/ $\SL i$ is isomorphic to  $\SB i^{++}$
for $i = 0,1, 2, \dots$.  See \emph{Figure~\ref{F:sequence2}.}
\end{corollary}

\begin{corollary}\label{C:main3}
There is an  infinite semidistributive lattice $L_i$ such that 
the congruence lattice of\/ $\SL i$ is isomorphic to  $\SB i \dotplus \SB 2$
for $i = 0,1, 2, \dots$.  See \emph{Figure~\ref{F:sequence3}.}
\end{corollary}

\begin{definition}\label{D:iso}
In a lattice $L$, an interval $[a, b]$ is \emph{isolated}, if the following three conditions hold:
\begin{enumeratei}
\item $a \prec b$;
\item $a$ is doubly-irreducible;
\item $b$ is doubly-irreducible.
\end{enumeratei}
\end{definition}

For $|L| > 2$, let $\ab$ denote the sublattice of $L$
we obtain by deleting the elements $a$ and $b$.
For $|L| = 2$, let $\ab = \es$.


\begin{theorem}\label{T:main2}
Let $L$ be a finite semidistributive lattice that has
an isolated interval $[a,b]$.
Then $D = \Con L$  can be represented
as the congruence lattice of an \emph{infinite} semidistributive lattice $K$.
\end{theorem}

The isolated interval $[a, b]$ in $\nfive$ yields the following.

\begin{corollary}\label{C:main2-1}
There is an infinite semidistributive lattice $L$ such that 
the congruence lattice of $L$ is isomorphic to  $(\SB 2)_+$.
\end{corollary} 

The 4-element chain contains an isolated interval. Other semidistributive lattices with an isolated interval include
$\SL 6 = \SN 6$, $\SL 9$, and $\SL {10}$ from the list of lattices
that generate varieties covering the variety of the pentagon,
see, for instance,  P. Jipsen and H. Rose \cite{JR2016}.

In Section~\ref{S:Representing}, we also prove some related representation theorems.
%
%
%

\subsection*{Basic concepts and notation.}

The basic concepts and notation not defined in this note 
are in Part~I of the book \cite{CFL3} (freely available for download). 
In particular, for a join-irreducible element $a$ in a finite lattice, 
let $a_*$ denote the unique lower cover of $a$; we define $a^*$ dually.

\begin{figure}[p!]
\centerline{\includegraphics{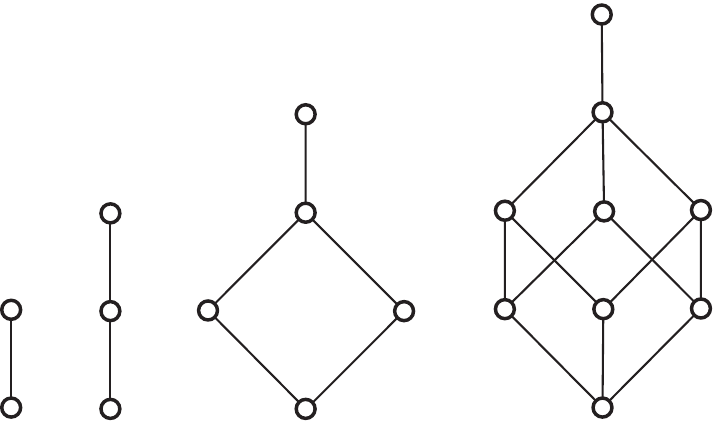}}
\caption{Finite distributive lattices represented in Theorem~\ref{T:main}}
\label{F:sequence1}

\bigskip

\bigskip

\centerline{\includegraphics{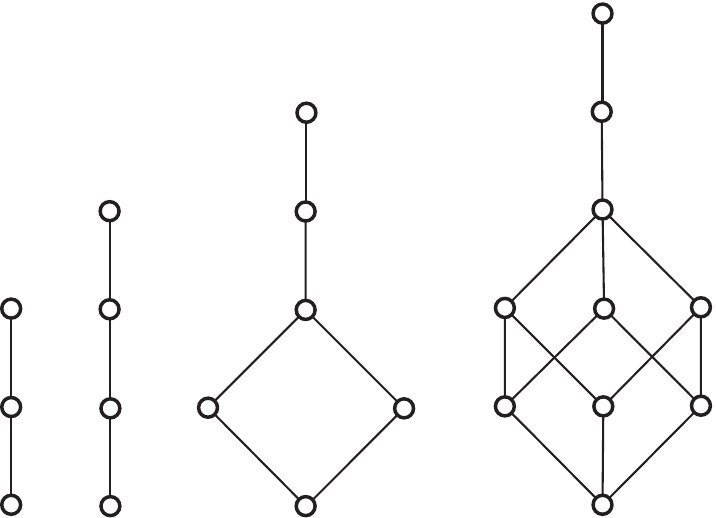}}
\caption{The finite distributive lattices represented in Corollary~\ref{C:main2}}
\label{F:sequence2}

\bigskip

\bigskip

\centerline{\includegraphics{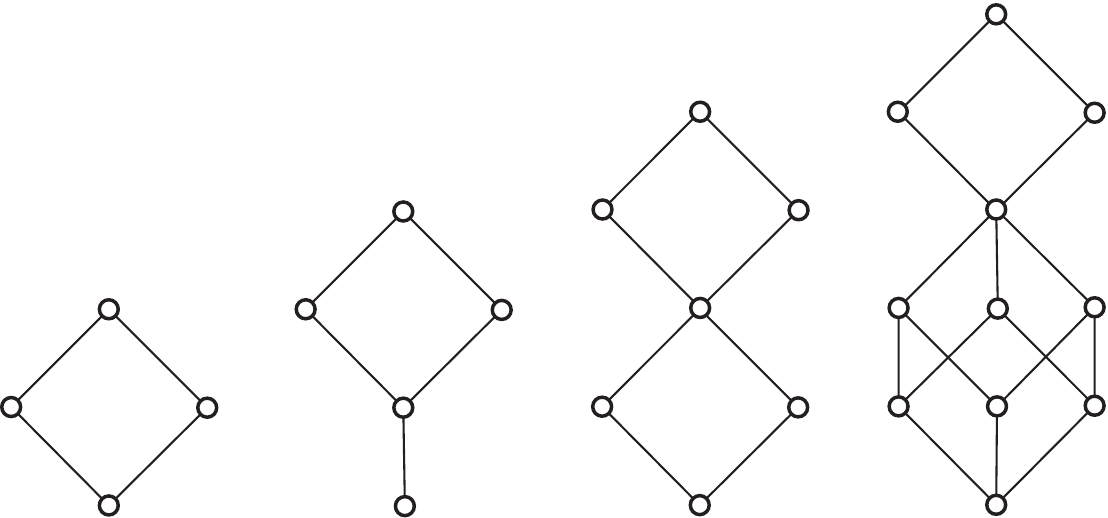}}
\caption{The finite distributive lattices represented in Corollary~\ref{C:main3}}
\label{F:sequence3}
\end{figure}

\section{Semidistributivity}\label{S:Semidistributivity}

A lattice $L$ is \emph{meet-semidistributive}, if the following implication holds:
\begin{equation}
    w = x \mm y  = x \mm z \text{ implies that } w = x \mm (y  \jj z) \text{ for $x,y,z,w \in L$}. 
\tag{\sdmeet}
\end{equation}
 
The following statement is well-known.

\begin{lemma}\label{L:prime}
Let $L$ be a meet-semidistributive $L$ with $u < v \in L$.
Then the set 
\[
   I_{u,v} = \setm{x \in L}{v \mm x = u}
\]
is an ideal in $\fil{u}$. 
Conversely, if $I_{u,v}$ is an ideal in $\fil{u}$ whenever $u < v$, then $L$ is meet-semidistributive.
\end{lemma}

Dually, $L$ is \emph{join-semidistributive}, if the following implication holds:
\begin{equation}
  u = x \jj y  = x \jj z \text{  implies that  }
   u = x \jj (y  \mm z) \text{ for $x,y,z,u \in L$}.\tag{\text{SD}${}_\jj$}
\end{equation}
A lattice $L$ is \emph{semidistributive} 
if it is both meet-semidistributive and join-semi\-distributive.

\section{Doubling an element}\label{S:Doubling}

In this section we review some familiar properties of the doubling construction;
see \cite{gG74} and \cite{Day94}.

Let $L$ be a lattice and let $u \in L$.
Define 
\begin{equation}\label{E:lr}
L[u] = (L - \set{u}) \uu \set{u_0, u_1} \end{equation} (disjoint union).

We order the set $L[u]$ by the relation $\leq_u$ as follows:
\begin{definition}\label{D:L[u]}
For $x,y \in L$, 
\begin{enumeratei}
\item let $(\set{u_0,u_1}, \leq_u)$ be isomorphic to the two-element chain;
\item let $x \leq_u y$ be equivalent to  $x \leq y$ for $x,y \nin \set{u_0,u_1}$;
\item let $u_i \leq_u y$ be equivalent to  $u \leq y$ for $i = 1,2$;
\item let $x \leq_u u_i$ be equivalent to  $x \leq u$ for $i = 1,2$.
\end{enumeratei}
\end{definition}

\begin{lemma}\label{L:splitting}\hfill

The following statements hold:
\begin{enumeratei}
\item the ordered set $L[u]$ is a lattice;
\item the element $u_0$ is meet-irreducible, the element $u_1$ is join-irreducible;
moreover, \text{$u_0 \prec u_1$} in $L[u]$;
\item for $u \nin \set{a,b,c}$, the join if $a \jj b = c$ in $L$ 
is the same as the join $a \jj b = c$ in $L[u]$ and dually;
\item for $u \neq a$, the meet $u_0 \mm a$ in $L[u]$ is the same 
as the meet $u \mm a$ in $L$ and dually;
\item the congruence $\bgm = \con{u_0, u_1}$ on $L[u]$ 
has only one nontrivial class, $\set{u_0,u_1}$;
moreover, $L[u] / \bgm \iso L$. 
\end{enumeratei}
\end{lemma}

%
%

The following lemma is well-known.  
It holds more generally for doubling intervals, but not
doubling arbitrary convex sets.

\begin{lemma}\label{L:lu}
If $L$ is a semidistributive lattice, then so is $L[u]$.
\end{lemma}

We define the congruence  $\bgm = \con{u_0, u_1}$ on the lattice $L[u]$.

\begin{lemma}\label{L:all}
The congruence $\bgm$ is an atom in the congruence lattice, $\Con{L[u]}$,
of the lattice $L[u]$.
\end{lemma}

\begin{proof}
Since $\bgm$ has only one nontrivial congruence class, namely, $\set{u_0,u_1}$,
it is obviously at atom in $\Con{L[u]}$. 
\end{proof}

We can start with a finite antichain $U \ci L$ and define 
\begin{equation}\label{E:lrU}
L[U] = (L - \set{u}) \uu \UUm{\set{u_0, u_1}}{u \in U}
\end{equation}
(disjoint unions). We define $\leq_U$ analogously.
The obvious analogue of Lemma~\ref{L:splitting} holds.

In the lattice $L[U]$, define the congruence  $\bgm_u = \con{u_0, u_1}$ for $u \in U$.
Then Lemma~\ref{L:all} holds in $L[U]$ for all $\bgm_u$.
 
\section{Representing the three-element chain and $(\SB 2)_{++}$}\label{S:Representing}

There are only two finite distributive lattices that are known not to be the congruence
lattice of a finite semidistributive lattice minimally \cite{JBN23}.  
The minimality refers to the fact that neither can they be a filter in $\Con L$, 
since a homomorphic image of a finite semidistributive lattice is semidistributive.
These are the three-element chain and the lattice $(\SB 2)_{++}$.

In Section~\ref{S:main}, 
we will utilize two constructions to represent each of these
as the congruence lattice of an infinite semidistributive lattice.

\begin{lemma}[Construction 1]\label{L:Construction1}
Let $F$ be an FN lattice. 
Define the lattice $F^+$ as the lattice $F$ with a new unit adjoined.
Then $F^+$ is an infinite semidistributive lattice 
and the congruence lattice of $F^+$ is the three-element chain, because $F$ has no largest element.
\end{lemma}

\begin{proof}
Obviously,  the lattice  $F^+$ is an infinite semidistributive lattice.
Indeed, the only nontrivial congruence of the lattice $F^+$
is the congruence with one nontrivial congruence class, namely, $F$.
\end{proof}

\begin{lemma}[Construction 2]\label{L:construction2}
Let $F$ be an FN lattice and let $u \in F$.
We double the element $u$ and claim that $F[u]$ is an infinite semidistributive lattice
and the congruence lattice of $F[u]$ is the three-element chain.
\end{lemma}

\begin{proof}
By Lemma~\ref{L:lu}, 
the lattice $F[u]$ is an infinite semidistributive lattice
and by Lemma~\ref{L:all}, 
the lattice $F[u]$ has only one nontrivial congruence, namely, $\bgm$,
so the congruence lattice of $F[u]$ is the three-element chain.  
(This uses $u \ne 0$ or $1$, but $F$ has neither.)
\end{proof}

In constructing a semidistributive lattice to represent $(\SB 2)_{++}$, we anticipate
Theorem~\ref{T:main2}.  
As in Corollary~\ref{C:main2-1}, we replace the isolated interval in $\nfive$ 
with an infinite, simple semidistributive lattice $F$.  Then double an element $u$ in 
the interval $F$ to obtain $(\SB 2)_{++}$ as the congruence lattice.

\section{Proving Theorem~\ref{T:main}}\label{S:main}

We define  $\SL 0$ as $F$. 
Since $\SB 0$ is the one-element chain, so $\SB 0^+ $ is the two-element chain, which 
is the congruence lattice of $F$.  Thus the statement of Theorem~\ref{T:main}
holds for $n = 0$.

Next we verify the statement of Theorem~\ref{T:main} for $n > 0$.
We prove the following, slightly stronger, result.

\begin{claim}\label{C:main}
There is an infinite semidistributive lattice $\SL n$ \emph{without bounds}
whose congruence lattice is isomorphic to  $\SB n^+$ for $n > 0$.
\end{claim}

\begin{proof}
First, we note that the lattice $F$ has antichain of any finite size,
for instance, contained in the countably infinite antichain $\set{b_0, b_1, \dots}$
using the notation in R.~Freese and J.\,B. Nation \cite{FN21}.

We start with an antichain $U \ci F$ of $n > 0$ elements and define $\SL n = F[U]$.
For every $V \ci U$, define 
\[
   \bgm[V] = \UUm{\bgm_v}{v \in V},
\]
a congruence of $\SL n = F[U]$.
The congruences $\setm{\bgm[V]}{V \ci U}$
form a sublattice~$C$ isomorphic to $\SB n$ of~$\Con F[U]$, the congruence lattice of $F[U]$.
The unit element of~$C$ is  $\bgm[U]$ and $F[U]/\bgm[U] \iso F$.
The zero element of  $C$ is $\bgm[\es]$ and $\bgm[\es] = \zero_{F[U]}$.
It follows that  the congruence lattice of $F[U]$ consists of $\unit_{F[U]}$ and $C$,
so it is isomorphic to  $\SB n^+$, that is $\SB n$ with a new unit adjoined.
\end{proof}

We have obtained a stronger form of Theorem~\ref{T:main}.

\begin{corollary}\label{C:strong}
There is an infinite semidistributive lattices without bounds $\SL i$ such that 
the congruence lattice of \,$\SL i$ is isomorphic to  $\SB i^+$
for $n = 0,1, 2, \dots$.
\end{corollary}
 
Observe that Corollary~\ref{C:main2} 
follows from Theorem~\ref{T:main} and Lemma~\ref{L:Construction1}.
 
To verify Corollary~\ref{C:main3}, we start with $L_0 = F^+_+$.
We obtain the lattice $L_i$ by~replacing $F$ with $F[U]$ as in the proof of Claim~\ref{C:main}.

\section{Proving Theorem~\ref{T:main2}}\label{S:main2}

Let $L$ be a finite semidistributive lattice, $D = \Con L$,
and let $[a, b]$ be an isolated interval of~$L$.
If $L$ has $2$ elements, the result is trivial, 
so we can assume that $L$ has~$3$ or more elements.
We denote $L - \set{a,b}$ by $\ab$;  it is a sublattice of $L$
because $a$ and $b$ are doubly-irreducible.

The lattice $\ab$ has a natural partition into the sets
\begin{align}\label{Eq:partition}
P &= \setm{x < a}{x \in\ab} = \id {a_*},\notag\\
Q &= \setm{x > b}{x \in \ab} \,= \fil{b^*},\\
R&= \setm{x \nless a}
{x \in\ab} \uu \setm{x \not> b}{x \in\ab}.\notag
\end{align}

We shall construct an infinite semidistributive lattice $K$ such that 
$D$ is the congruence lattice of $K$.

If $L$ has $2$ elements, then $K = F$. So we can assume that 
$L$ has more than~$2$ elements.

Let $K =  \ab \uu F$, where  $F$ is an FN lattice 
(a simple, infinite, semidistributive lattice). 
Note that the lattice $K$ has a natural partition into the sets 
(\ref{Eq:partition}) and the set $F$, see Figure \ref{F:partition},
see 
\begin{figure}[htb]
\centerline{\includegraphics{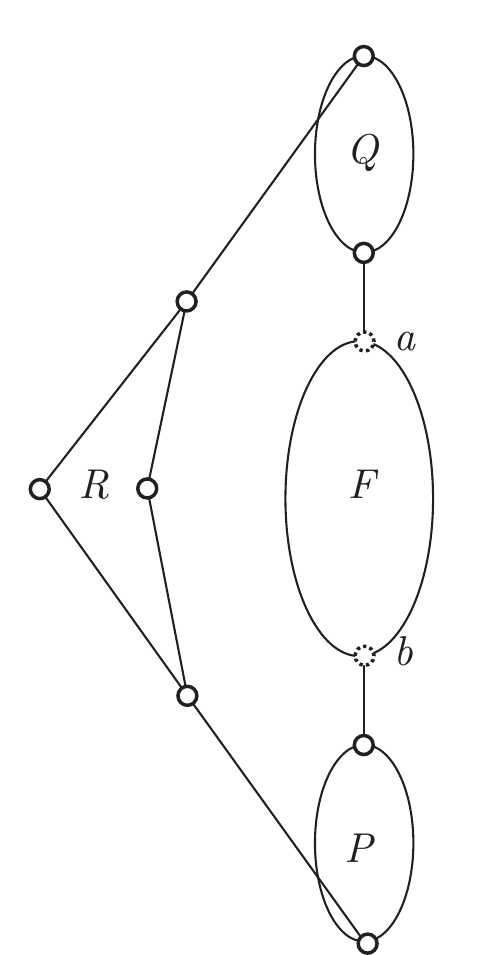}}
\caption{The partitioning of the lattice $K$}
\label{F:partition}
\end{figure}

\begin{definition}\label{D:bin}
We define a binary relation $<_K$ on $K$ by the following rules:
\begin{enumeratei}
\item for $x, y \in  \ab$, let $x <_K y$ 
be equivalent to  $x < y$ in $L$.
\item for $x \in \ab$ and $y \in F$, let $x <_K y$ 
be equivalent to  $x \leq a$ in $L$.
\item for $x \in F$ and $y \in\ab$, let $x <_K y$ 
be equivalent to  $b \leq y$ in $L$.
\end{enumeratei}
\end{definition}

It is easy to see that
$K$ is an ordered set with respect to the binary relation $<_K$.

\begin{lemma}\label{L:lattice}
The ordered set $K$ is a lattice and $\ab$ is a sublattice of $K$.
\end{lemma}

\begin{proof}
$\ab$ is a sublattice of $K$ because $a$ and $b$ are doubly irreducible 
(see \ref{D:bin}.(ii) and (iii)).
To prove that $K$ is a lattice is tedious, so we skip most of it.
Here is one step: Let $x \in \ab$, $x \nleq a$ in $L$ and $y \in F$.
Let $z = x \jj b$ in $L$. Then $z = x \jj y$ in~$K$.
\end{proof}

\begin{lemma}\label{L:semi}
The lattice $K$ is semidistributive.
\end{lemma}

\begin{proof}
We utilize Lemma~\ref{L:prime}. So for $u < v \in K$, we have to verify that
the set 
\[
   I_{u,v} = \setm{x \in K}{v \mm x = u}
\]
is an ideal in $\fil{u}$. 

If $u,v \in \ab$, then the statement is true by Lemma~\ref{L:prime}
because $L$ and $\ab$ are semidistributive.
Thus if join-semidistibutivity fails, then $u$ or $v$---or both---are in $F$.
Both cannot be in $F$, because it is semidistributive.
But clearly there can be no failure with only one of $u$ and $v$ in $F$;
view Figure~\ref{F:partition}.

The meet-semidistributivity of $K$ is shown dually.
\end{proof}

Finally, let $\bga \in \Con L$. 
We associate with $\bga$ the binary relation $\bgaab$ on $\ab$
defined as follows. 

\begin{definition}\label{D:xx}
If $\cng a=b(\bga)$, then the congruence classes of $\bgaab$ on $\ab$
are the same as the congruence classes of $\bga$ on $L$
not containing $a$ (or $b$) and $F$ 
is one more congruence class.      

If $\ncng a=b(\bga)$, then the congruence classes of $\bgaab$ on $\ab$
are the same as the congruence classes of $\bga$ on $L$
not containing $a$ or $b$) and and every element of~$F$ 
is a trivial congruence class. 
\end{definition}

\begin{lemma}\label{L:cong}
The map $\bga \to \bgaab$ establishes an isomorphism between $D = \Con L$
and $\Con \ab$.
\end{lemma}

Lemmas~\ref{L:lattice}--\ref{L:cong} complete the proof of  Theorem~\ref{T:main2}.

\end{document}